\documentclass[11pt,reqno]{amsart}

\newtheorem{theorem}{Theorem}[section]
\newtheorem{lemma}[theorem]{Lemma}
\newtheorem{proposition}[theorem]{Proposition}
\newtheorem{corollary}[theorem]{Corollary}

\theoremstyle{definition}

\numberwithin{equation}{section}

\begin{document}

\baselineskip=15.5pt

\title[Equivariant bundles and adapted connections]{Equivariant bundles and adapted 
connections}

\author[I. Biswas]{Indranil Biswas}

\address{School of Mathematics, Tata Institute of Fundamental
Research, Homi Bhabha Road, Mumbai 400005, India}

\email{indranil@math.tifr.res.in}

\author[A. Paul]{Arjun Paul}

\address{School of Mathematics, Tata Institute of Fundamental 
Research, Homi Bhabha Road, Mumbai 400005, India} 

\email{apmath90@math.tifr.res.in}

\author[A. Saha]{Arideep Saha}

\address{School of Mathematics, Tata Institute of Fundamental 
Research, Homi Bhabha Road, Mumbai 400005, India} 

\email{arideep@math.tifr.res.in}

\subjclass[2010]{32M10, 53B15, 37F75.}

\keywords{Equivariant bundle, $G$--connection, adapted connection, holomorphic foliation.}

\begin{abstract}
Given a complex manifold $M$ equipped with a holomorphic action of a connected complex
Lie group $G$, and a holomorphic principal $H$--bundle $E_H$ over $X$ equipped with a 
$G$--connection $h$, we investigate the connections on the principal $H$--bundle $E_H$ 
that are (strongly) adapted to $h$. Examples are provided by holomorphic principal 
$H$--bundles equipped with a flat partial connection over a foliated manifold.
\end{abstract}

\maketitle

\section{Introduction}

Let $X$ be a complex manifold, $G$ a connected complex Lie group and
$\rho\, :\, G\times X\, \longrightarrow\, X$ a holomorphic action of $G$ on $X$.
The Lie algebra of $G$ is denoted by $\mathfrak g$.
Let $p\, :\, E_H\, \longrightarrow\, X$ be a holomorphic principal $H$--bundle, where
$H$ is a complex Lie group. A $G$--connection on $E_H$ is a $\mathbb C$--linear
map $h\, :\, {\mathfrak g}\, \longrightarrow\, H^0(E_H, \, TE_H)^H$ such that for
every $v\, \in\, \mathfrak g$, the vector field $dp\circ h(v)$ on $X$ coincides with the
one defined by $v$ using the above action $\rho$ (see Section \ref{se2.1}). In \cite{BP},
$G$--connections were investigated, in particular, a criterion was given for the
existence of a $G$--connection.

Here we continue the investigations of $G$--connections. More precisely, we study the 
interactions of $G$--connections on $E_H$ with the holomorphic connections on the 
principal $H$--bundle $E_H$. There are two possible compatibility conditions between 
them which are called ``adapted'' and ``strongly adapted'' (see Section \ref{se3.1}). 
To explain these conditions, if $h$ is given by a holomorphic action $\rho_E$ of $G$ 
on $E_H$, then a holomorphic connection $\eta$ on the principal
$H$--bundle $E_H$ is adapted to $h$ if and only 
if $\eta$ is preserved by $\rho_E$; such an adapted connection $\eta$ is called strongly 
adapted if the image of the homomorphism $h$ is contained in the horizontal subbundle 
of $TE_H$ for the connection $\eta$.

The property of a holomorphic connection $\eta$ on a holomorphic principal $H$--bundle 
$E_H$ that it is strongly adapted to a $G$--connection $h$ on $E_H$ can also be 
formulated in the context of foliated manifolds and principal $H$--bundles on them equipped 
with a flat partial connection; the details are in Section \ref{se-f}.

\section{Preliminaries}

\subsection{Atiyah bundle}

Let $H$ be a complex Lie group. Its Lie algebra will be denoted by
$\mathfrak h$. Let $X$ be a connected complex manifold and
\begin{equation}\label{e-1}
p\, :\, E_H\, \longrightarrow\, X
\end{equation}
a holomorphic principal $H$--bundle over $X$. This means that $E_H$
is a complex manifold equipped with a holomorphic right action of $H$
$$
a\, :\, E_H\times H\, \longrightarrow\, E_H
$$
such that
\begin{itemize}
\item $p\circ a\,=\, p\circ p_{E_H}$, where $p_{E_H}$ is the projection of
$E_H\times H$ to $E_H$, and

\item the map $(p_{E_H},\, a)\, :\, E_H\times H\, \longrightarrow\, E_H\times_X E_H$
is an isomorphism.
\end{itemize}
Note that the first condition means that the action of $H$ takes a fiber of $p$ to
itself, so the image of the map $(p_{E_H},\, a)$ is contained in the fiber product
$E_H\times_X E_H$. The second condition above means that the action of $H$ on a
fiber of $p$ is free and transitive.

The adjoint bundle for $E_H$
$$
\text{ad}(E_H)\, :=\, E_H\times^H {\mathfrak h}\, \longrightarrow\, X
$$
is the holomorphic vector bundle over $X$ associated to $E_H$ for the adjoint action of $H$
on the Lie algebra $\mathfrak h$.

The holomorphic tangent (respectively, cotangent) bundle of a complex manifold $Y$ 
will be denoted by $TY$ (respectively, $T^*Y$). The tangent bundle of a real manifold 
$Y$ will be denoted by $T^{\mathbb R}Y$.

The {\it Atiyah bundle} for $E_H$
$$
\text{At}(E_H)\, :=\, (TE_H)/H \, \longrightarrow\, E_H/H\,=\, X
$$
is a holomorphic vector bundle over $X$ whose rank is $\dim X+\dim {\mathfrak h}$;
see \cite{At}. Let
$$T_{E_H/X}\, \subset\, TE_H$$ be the relative tangent bundle for the projection
$p$ in \eqref{e-1}. The subbundle
$$
(T_{E_H/X})/H\, \subset\, (TE_H)/H\,=\, \text{At}(E_H)
$$
is identified with the adjoint vector bundle $\text{ad}(E_H)$. This identification
is a consequence of the isomorphism of $T_{E_H/X}$ with the trivial vector
bundle $E_H\times \mathfrak h\, \longrightarrow\, E_H$ given by the action of $H$
on $E_H$. Therefore, the short exact sequence
$$
0\, \longrightarrow\, T_{E_H/X}\, \longrightarrow\, TE_H\, \stackrel{dp}{\longrightarrow}
\, p^*TX \, \longrightarrow\, 0\, ,
$$
where $dp$ is the differential of $p$, produces a short exact sequence on $X$
\begin{equation}\label{e2}
0\, \longrightarrow\, \text{ad}(E_H)\, \longrightarrow\, \text{At}(E_H)\,
\stackrel{dp}{\longrightarrow} \, TX \, \longrightarrow\, 0\, ,
\end{equation}
which is known as the \textit{Atiyah exact sequence} for $E_H$. For simplicity, we have used the same notation $dp$ for the differential $TE_H \,\longrightarrow\, p^*TX$
over $E_H$ as well as its descent $\text{At}(E_H) \,\longrightarrow\, TX$ to $X$. A holomorphic connection
on $E_H$ is a holomorphic homomorphism
\begin{equation}\label{eta}
\eta\, :\, TX\, \longrightarrow\,\text{At}(E_H)
\end{equation}
such that $(dp)\circ\eta\,=\, \text{Id}_{TX}$, where $dp$ is the
homomorphism in \eqref{e2}. For a holomorphic connection $\eta$
on $E_H$, the homomorphism
$$
\bigwedge\nolimits^2 TX\, \longrightarrow\,\text{ad}(E_H)\, ,\ \
v\otimes w - w\otimes v \, \longmapsto\, 2([\eta(v),\, \eta(w)]- \eta([v,\, w]))\, ,
$$
where $v$ and $w$ are locally defined holomorphic sections of $TX$, produces a
holomorphic section of $(\bigwedge^2 T^*X)\otimes \text{ad}(E_H)$. This
holomorphic section of $(\bigwedge^2 T^*X)\otimes \text{ad}(E_H)$ is called the
\textit{curvature} of the connection $\eta$.

The vector bundle $TE_H\otimes p^*(TX)^*$ on $E_H$ has a natural action of $H$ given by
the action of $H$ on $TE_H$ and the tautological action of $H$ on $p^*(TX)^*$.
We note that a holomorphic connection on $E_H$ is an $H$--invariant holomorphic
section of $TE_H\otimes p^*(TX)^*$.

\subsection{$G$--connections on $E_H$}\label{se2.1}

Let $G$ be a connected complex Lie group; its Lie algebra will be denoted by $\mathfrak g$.
The identity element of $G$ will be denoted by $e$. Let
\begin{equation}\label{e6}
\rho\,:\, G\times X\, \longrightarrow\, X
\end{equation}
be a holomorphic action of $G$ on $X$. Consider the holomorphic homomorphism
$$
\rho'\, :\, \text{At}(E_H)\oplus (X\times{\mathfrak g})\,\longrightarrow\, TX\, ,
\ \ (v\, ,w)\, \longmapsto\, dp(v)-d'\rho(w)\, ,
$$
where $dp$ is the homomorphism in \eqref{e2}, and
\begin{equation}\label{dpr}
d'\rho\, :\, X\times{\mathfrak g}\, \longrightarrow\, TX\, ,\ \
(x,\, v)\, \longmapsto\, (d\rho)(e,x)(v,0)\, ,
\end{equation}
with $(d\rho)(e,x)\, :\, {\mathfrak g}\oplus T_xX\,\longrightarrow\, T_xX$ being the
differential of $\rho$ at $(e,\,x)\,\in\, G\times X$. Define the subsheaf
\begin{equation}\label{e8}
\text{At}_\rho(E_H)\, :=\,
(\rho')^{-1}(0)\, \subset\, \text{At}(E_H)\oplus (X\times{\mathfrak g})\, .
\end{equation}
Since the differential $dp$ is surjective, it follows that $\rho'$ is surjective.
This implies that $\text{At}_\rho(E_H)$ is a holomorphic subbundle of
$\text{At}(E_H)\oplus (X\times{\mathfrak g})$. The
vector bundle $\text{At}_\rho(E_H)$ fits in a
commutative diagram with exact rows
\begin{equation}\label{e9}
\begin{matrix}
0 & \longrightarrow & \text{ad}(E_H) & {\longrightarrow}& \text{At}_\rho
(E_H) & \stackrel{q}{\longrightarrow}& X\times{\mathfrak g} & \longrightarrow & 0\\
&& \Vert && ~ \Big\downarrow J && ~\,~\,~\,~\, \Big\downarrow d'\rho\\
0 & \longrightarrow & \text{ad}(E_H) & \longrightarrow & \text{At}
(E_H) & \stackrel{dp}{\longrightarrow}& TX & \longrightarrow & 0\\
\end{matrix}
\end{equation}
where $J$ (respectively, $q$) is given by the projection of
$\text{At}(E_H)\oplus (X\times{\mathfrak g})$ to
$\text{At}(E_H)$ (respectively, $X\times{\mathfrak g}$). (See \cite{BP}.)

A \textit{holomorphic} $G$--\textit{connection} on $E_H$ is a
holomorphic homomorphism of vector bundles
\begin{equation}\label{h}
h\, :\, X\times{\mathfrak g}\,\longrightarrow\,\text{At}_\rho(E_H)
\end{equation}
such that $q\circ h\,=\, \text{Id}_{X\times{\mathfrak g}}$, where $q$
is the homomorphism in \eqref{e9}. The curvature
of a $G$--connection $h$
$$
(s,\, t)\, \longmapsto\, [h(s),\, h(t)] - h([s,\, t])
$$
is a holomorphic section
\begin{equation}\label{b}
{\mathcal K}(h)\, \in\, H^0(X,\, \text{ad}(E_H)\otimes \bigwedge\nolimits^2
(X\times{\mathfrak g})^*)\, =\, H^0(X,\, \text{ad}(E_H))\otimes\bigwedge\nolimits^2
{\mathfrak g}^*\, .
\end{equation}

We will give examples of $G$--connection.

Let $a\, :\, E_H\times H\, \longrightarrow\, E_H$ be the action of $H$ on the
principal $H$--bundle $E_H$.

A $G$--action on the principal bundle $E_H$ is a holomorphic action
of $G$ on the total space of $E_H$
\begin{equation}\label{rhoE}
\rho_E\, :\, G\times E_H\, \longrightarrow\, E_H
\end{equation}
such that
\begin{enumerate}
\item{} $p\circ\rho_E\,=\, \rho\circ (\text{Id}_G\times p)$, where $p$ and $\rho$
are the maps in \eqref{e-1} and \eqref{e6} respectively, and

\item{} $\rho_E\circ (\text{Id}_G\times a)\,=\, a\circ (\rho_E\times \text{Id}_H)$
as maps from $G\times E_H\times H$ to $E_H$ (this condition means that the actions of
$G$ and $H$ on $E_H$ commute).
\end{enumerate}
An equivariant principal $H$--bundle is a holomorphic principal $H$--bundle with
a $G$--action.

Let $\rho_E\, :\, G\times E_H\, \longrightarrow\, E_H$ be a $G$--action on $E_H$.
Consider the homomorphism 
$$
\widetilde{h}\, :\, E_H\times{\mathfrak g}\,\longrightarrow\,TE_H
$$
given by the differential $d\rho_E$ of the action $\rho_E$; more precisely,
$\widetilde{h}(z, v)\,=\, d\rho_E(e, z)(v,0)$, so $\widetilde{h}$ is the homomorphism
in \eqref{dpr} when $X$ is substituted by $E_H$.
Since the actions of $G$
and $H$ on $E_H$ commute, this homomorphism $\widetilde{h}$ produces a $G$--connection
\begin{equation}\label{gcr}
h_0\, :\, X\times{\mathfrak g}\,\longrightarrow\,\text{At}_\rho(E_H)
\end{equation}
on $E_H$; the curvature of this $G$--connection $h_0$ vanishes identically
\cite[Lemma 4.1]{BP}.

Let $Y$ be a connected compact complex manifold such that $TY$ is holomorphically trivial.
Then $Y$ is holomorphically
isomorphic to $G/\Gamma$, where $G$ is a connected complex Lie group
and $\Gamma\, \subset\, G$ is a cocompact lattice \cite{Wa}; in fact, $G$ is the
connected component, containing the identity element, of the group of all holomorphic
automorphisms of $Y$. Consider the left--translation action of $G$ on $G/\Gamma\,=\, Y$. A
$G$--connection on a holomorphic principal $H$--bundle $E_H$ on $Y$
is an usual holomorphic connection on the principal $H$--bundle.

\subsection{Distributions under a flow}

Let $Y$ be a connected $C^\infty$ manifold and
$$
{\mathcal D}\, \subset\, T^{\mathbb R}Y
$$
a $C^\infty$ subbundle. In other words, $\mathcal D$ is a distribution on $Y$. The
fiber of $\mathcal D$ over any point $z\, \in\, Y$ will be denoted by ${\mathcal D}_z$.

Let $\xi$ be a $C^\infty$ vector field on $Y$. Given any point $x\, \in\, Y$,
there is an open neighborhood $x\, \in\, U_x\, \subset\, Y$ and an open
interval $0\, \in\, I_x\, \subset\, \mathbb R$, such that $\xi$ integrates to a
flow
$$
\Phi_x\, :\, U_x\times I_x\, \longrightarrow\, Y\, .
$$
For any $t\, \in\, I_x$, define
$$
\Phi_{x,t}\, :\, U_x\, \longrightarrow\, Y\, ,\ \ z\, \longmapsto\, \Phi_x(z,t)\, .
$$

\begin{lemma}\label{lem1}
The following two are equivalent:
\begin{enumerate}
\item For every $x\, \in\, Y$ and $z\, \in\, U_x$ as above,
$$
(d\Phi_{x,t})(z)({\mathcal D}_z)\,=\, {\mathcal D}_{\Phi_{x,t}(z)}\, ,
$$
where $d\Phi_{x,t}(z)\, :\, T^{\mathbb R}_z Y\, \longrightarrow\,
T^{\mathbb R}_{\Phi_{x,t}(z)}Y$ is the differential of the map $\Phi_{x,t}$ at $z$.

\item $[\xi,\, {\mathcal D}]\, \subset\, {\mathcal D}$.
\end{enumerate}
\end{lemma}

\begin{proof}
Let ${\mathcal W}$ denote the space of all $C^\infty$ $1$--forms on $Y$ 
that vanish on $\mathcal D$. The first statement is equivalent to the statement
that
\begin{equation}\label{e3}
L_\xi(w)\, \in\, {\mathcal W}\ \ \ \forall~\ w\, \in\, {\mathcal W}\, ,
\end{equation}
where $L_\xi$ denotes the Lie derivative with respect to the vector field $\xi$.

First assume that
\begin{equation}\label{e4}
[\xi,\, {\mathcal D}]\, \subset\, {\mathcal D}\, .
\end{equation}
To prove that
\eqref{e3} holds, take any $w\, \in\, {\mathcal W}$ and any $C^\infty$ section
$\theta$ of $\mathcal D$. We have
$$
(L_\xi(w))(\theta) \,=\, \xi(w(\theta)) - w(L_\xi\theta) \,=\,
\xi(w(\theta)) - w([\xi,\, \theta])\, .
$$
Now, $w(\theta)\,=\, 0$, and $[\xi,\, \theta]$ is section of $\mathcal D$ by
\eqref{e4}. Hence $(L_\xi(w))(\theta) \,=\, 0$, which implies that
\eqref{e3} holds.

Now assume that \eqref{e3} holds. To prove \eqref{e4}, let
$\theta$ be any $C^\infty$ section of $\mathcal D$. Take any $w\, \in\, \mathcal W$. 
We have
$$
w([\xi,\, \theta])\,=\, w(L_\xi\theta)\,=\, \xi(w(\theta)) - (L_\xi w)(\theta)\, .
$$
Now, $w(\theta)\,=\, 0$, and also $(L_\xi w)(\theta)\,=\, 0$ because $L_\xi w\,\in\,
\mathcal W$ by \eqref{e3}. Hence \eqref{e4} holds.
\end{proof}

\section{Connections and (strongly) adapted connections}

\subsection{Definitions}\label{se3.1}

Let $E_H$ be a holomorphic principal bundle on $X$ equipped with a holomorphic
connection 
$$
\eta\, :\, TX\, \longrightarrow\,\text{At}(E_H)
$$
(see \eqref{eta}). Since $\text{At}(E_H)\,=\, (TE_H)/H$, the image of $\eta$
is a holomorphic distribution on $E_H$; it is known as the \textit{horizontal
distribution} for the connection $\eta$.

As before, a connected complex Lie group $G$ acts holomorphically on $X$.

Given a holomorphic $G$--connection
$h\, :\, X\times{\mathfrak g}\,\longrightarrow\,\text{At}_\rho(E_H)$ on $E_H$ (see
\eqref{h}), the connection $\eta$ is said to be \textit{adapted} to $h$ if
\begin{equation}\label{a1}
[J\circ h(X\times \{v\}),\, \eta(TX)]\, \subset\, \eta(TX)\ \ \forall\ v\, \in\,
\mathfrak g\, ,
\end{equation}
where $J$ is the homomorphism in \eqref{e9}. Note that a $C^\infty$ section of $\text{At}(E_H)$
defines a $H$--invariant vector field on $E_H$ of type $(1,\, 0)$.

The connection $\eta$ is said to be
\textit{strongly adapted} to $h$ if it is adapted to $h$, and furthermore
\begin{equation}\label{Jc}
\text{image}(J\circ h) \, \subset\, \text{image}(\eta)\, .
\end{equation}

We will now give examples to show that the conditions in \eqref{a1} and \eqref{Jc}
are independent.

Consider the trivial action of the multiplicative group ${\mathbb C}^*\,=\,
{\mathbb C}\setminus\{0\}$ on $X$. Let $E$ be a holomorphic principal
$\text{GL}(r,{\mathbb C})$--bundle on $X$ admitting a holomorphic connection, for example
$E$ can be the trivial holomorphic principal $\text{GL}(r,{\mathbb C})$--bundle
$X\times \text{GL}(r,{\mathbb C})$ on $X$.
The center of $\text{GL}(r,{\mathbb C})$ is identified with ${\mathbb C}^*$ by sending
any $c\, \in\, {\mathbb C}^*$ to $c\cdot \text{Id}_{{\mathbb C}^r}\, \in\,
\text{GL}(r,{\mathbb C})$. Using this identification, the action of the center of
$\text{GL}(r,{\mathbb C})$ on $E$ produces an action of ${\mathbb C}^*$ on
$E$. Since ${\mathbb C}^*$ is in the center of $\text{GL}(r,{\mathbb C})$, the actions
of ${\mathbb C}^*$ and $\text{GL}(r,{\mathbb C})$ on $E$ commute.
If $E'$ is the vector bundle of rank $r$ associated to $E$ by the
standard representation of $\text{GL}(r,{\mathbb C})$, then this action of ${\mathbb C}^*$ on
$E$ corresponds to the action of ${\mathbb C}^*$ on $E'$ as scalar multiplications.
Let $h$ be the holomorphic ${\mathbb C}^*$--connection on $E$ given by this
action of ${\mathbb C}^*$ on $E$ (see \eqref{gcr}). Any holomorphic
connection on the principal $\text{GL}(r,{\mathbb C})$--bundle $E$ is adapted to $h$. But
\eqref{Jc} fails for every holomorphic connection on $E$.

Now take $X\,=\, {\mathbb C}^2$ and $G\,=\, \mathbb C\,=\, H$. Let $E_H$ be the 
trivial principal $\mathbb C$--bundle ${\mathbb C}^2\times {\mathbb C}\, 
\longrightarrow\, {\mathbb C}^2$. Take $\rho$ to be the action of $\mathbb C$ on 
${\mathbb C}^2$ defined by $$(z,\, (x,\, y))\, \longmapsto\, (x+z,\, y)\, , \ \ 
z\,\in\, \mathbb C\, , \ \ (x,\, y)\, \in\, {\mathbb C}^2\, .$$ This action of 
$\mathbb C$ on $X$ and the trivial action of $\mathbb C$ on $\mathbb C$ together 
define an action of $\mathbb C$ on $E_H\,=\, X\times\mathbb C$. Let $h$ be the 
holomorphic $\mathbb C$--connection on $E_H$ associated to this action of $\mathbb 
C$ on $E_H$ (see \eqref{gcr}). Let $D$ be the holomorphic connection on the 
principal $H$--bundle $E_H$ defined by $f\, \longmapsto\, df+ xf\cdot dy$, where $f$ 
is any holomorphic function on ${\mathbb C}^2$ (holomorphic sections of $E_H$ are 
holomorphic functions) and $d$ denotes the standard de Rham differential. Then 
\eqref{Jc} holds while \eqref{a1} fails.

\subsection{Equivariant bundles and adaptable connections}\label{s3.2}

As in \eqref{rhoE}, take a $G$--action $\rho_E$ on $E_H$. As mentioned earlier, there
is a natural $G$--connection on $E_H$
\begin{equation}\label{h0}
h_0\, :\, X\times{\mathfrak g}\,\longrightarrow\,\text{At}_\rho(E_H)
\end{equation}
corresponding to $\rho_E$.

Let $p_X\, :\, G\times X\, \longrightarrow\, X$ be the natural projection. The action
$\rho_E$ produces a holomorphic isomorphism of principal $H$--bundles
\begin{equation}\label{beta}
\beta\, :\, p^*_XE_H\, \longrightarrow\, \rho^*E_H\, ,\ \ \beta(g, x)(z)\,=\,
\rho_E(g, z) \ \ \forall\ g\,\in\, G,\;\; x\,\in\, X,\;\; z\, \in\, (E_H)_x\, ,
\end{equation}
where $\rho$ is the map in \eqref{e6}.

For any $g\, \in\, G$, let
$$
j_g\, :\, X\, \longrightarrow\,G\times X\, ,\ \ \ x \longmapsto\, (g,\, x)
$$
be the embedding. For all $g\, \in\, G$, the isomorphism $\beta$ in \eqref{beta}
produces a holomorphic isomorphism of principal $H$--bundles
\begin{equation}\label{bg2}
\beta^g\, :\, E_H\, \longrightarrow\, (\rho\circ j_g)^*E_H\, , \ \ \
z\, \longmapsto\, \beta(g,\, x)(z)\,=\, \rho_E(g, z)\ \ \forall \ x\,\in\, X, \
z\,\in\, (E_H)_x\, .
\end{equation}
The map from the holomorphic connections on $E_H$ to the holomorphic connections on
$(\rho\circ j_g)^*E_H$ induced by the above isomorphism $\beta^g$ will be denoted by
$\beta^g_*$; note that $\beta^g_*$ is a bijection.

\begin{proposition}\label{prop1}
A holomorphic connection $\eta$ on $E_H$ is adapted to the $G$--connection $h_0$ 
in \eqref{h0} associated to $\rho_E$ if and only if for all $g\, \in\, G$,
\begin{equation}\label{hold}
(\rho \circ j_g)^*\eta\,=\, \beta^g_*(\eta)
\end{equation}
(both are connections on the principal $H$--bundle $(\rho \circ j_g)^*E_H$).
\end{proposition}

\begin{proof}
First assume that $\eta$ is adapted to $h_0$. Take any $v\, \in\, \mathfrak g$. The
flow on $E_H$ generated by $v$ sends any $t\, \in\, \mathbb R$ to the biholomorphism
$$
F_t\,:\, E_H\, \longrightarrow\, E_H\, ,\ \ \
z\, \longmapsto\, \rho_E(\exp(tv),\, z)\, .
$$
Note that $F_t$ coincides with $\beta^{\exp(tv)}$ constructed in \eqref{bg2}.
Consider the $H$-invariant distribution $$D^\eta \,:=\, \text{image}(\eta)\, \subset\, TE_H\, .$$
Its fiber over any point $z\, \in\, E_H$ will be denoted by $D^\eta_z$.
Since $\eta$ is adapted to $h_0$, from Lemma \ref{lem1} it follows that
\begin{equation}\label{fl}
(dF_t)(z)(D^\eta_z) \, =\, D^\eta_{F_t(z)}
\end{equation}
for all $z\, \in\, E_H$ and $t\, \in\, {\mathbb R}$, where
$(dF_t)(z)\, :\, T_zE_H\, \longrightarrow\, T_{F_t(z)}E_H$ is the differential of
the map $F_t$. Since the subset $\{\exp(tv)\}_{v\in {\mathfrak g}, t\in {\mathbb R}}
\, \subset\, G$ is
dense in the analytic topology (recall that $G$
is connected), and also $F_t\,=\, \beta^{\exp(tv)}$, from \eqref{fl}
we conclude that \eqref{hold} holds for all $g\, \in\, G$.

Now assume that \eqref{hold} holds for all $g\, \in\, G$. This implies that
\eqref{fl} holds for all $z\, \in\, E_H$ and $t\, \in\, {\mathbb R}$. Consequently, from
Lemma \ref{lem1} we conclude that $\eta$ is adapted to $h_0$.
\end{proof}

Take any point $x\, \in\, X$. Define
$$
\rho_x\, :\, G\, \longrightarrow\, X\, ,\ \ \ g\, \longmapsto\,
\rho\circ j_g(x)\,=\,\rho (g,\, x)\, .
$$
Consider the map
$$
\rho_{E,x}\, :\, G\times (E_H)_x\, \longrightarrow\, \rho^*_x E_H\, ,\ \ \
(g,\, z)\,\longmapsto\, \rho_E(g,\, z)\, .
$$
Since this $\rho_{E,x}$ is $H$--equivariant (recall that the actions of $G$ and $H$
on $E_H$ commute), it identifies the pulled back principal
$H$--bundle $\rho^*_xE_H$ with the trivial principal
$H$--bundle $G\times (E_H)_x\, \longrightarrow\, G$. Let $D^0_x$ be the holomorphic
connection on the principal $H$--bundle $\rho^*_x E_H$ induced by the trivial connection
on $G\times (E_H)_x$ using the above isomorphism $\rho_{E,x}$. Note that
$\rho^*_x E_H$ is identified with the restriction of $\rho^*E_H$ to $G\times\{x\}$,
because $\rho_x$ is the restriction of $\rho$ to $G\times\{x\}$. Therefore, 
$\rho^*\eta\vert_{G\times\{x\}}$ is also a connection
on $\rho^*_x E_H$.

\begin{proposition}\label{prop2}
A holomorphic connection $\eta$ on $E_H$ is strongly adapted to the $G$--connection $h_0$
in \eqref{h0} if and only if the following two hold:
\begin{enumerate}
\item For all $g\, \in\, G$,
$$
(\rho\circ j_g)^*\eta\,=\, \beta^g_*(\eta)\, .
$$

\item For every $x\, \in\, X$, the connection $D^0_x$ on $\rho^*_x E_H$ coincides with
the connection $\rho^*\eta\vert_{G\times\{x\}}$.
\end{enumerate}
\end{proposition}

\begin{proof}
First assume that $\eta$ is strongly adapted to $h_0$. Since
$\eta$ is adapted to $h_0$, Proposition \ref{prop1} says that $(\rho\circ j_g)^*\eta\,
=\, \beta^g_*(\eta)$ for all $g\, \in\, G$. The given condition \eqref{Jc} implies that
the connection $D^0_x$ coincides with $\rho^*\eta\vert_{G\times\{x\}}$.

The converse is similarly proved. Assume that the two statements in the
proposition hold. From Proposition \ref{prop1} we know that $\eta$ is
adapted to $h_0$. The second condition in the proposition implies that \eqref{Jc} holds. 
\end{proof}

\section{Criterion for adapted connection}

Let $\eta\, :\, TX\, \longrightarrow\, \text{At}(E_H)$ be a holomorphic
connection on $E_H$. Let
\begin{equation}\label{te}
\widetilde{\eta}\, :\, 
X\times{\mathfrak g}\, \longrightarrow\,\text{At}(E_H)\oplus (X\times{\mathfrak g})
\end{equation}
be the ${\mathcal O}_X$--linear homomorphism defined by
$$
(x,\, v)\, \longmapsto\, (\eta(d'\rho(x,v)),\, (x,\, v))\, ,
$$
where $d'\rho$ is the homomorphism in \eqref{dpr}. Since we have 
$(dp)\circ \eta\,=\, \text{Id}_{TX}$, where $dp$ is the homomorphism in
\eqref{e2}, it follows immediately that the image of $\widetilde{\eta}$ is contained in
$\text{At}_\rho(E_H)\, :=\,(\rho')^{-1}(0)$ (see \eqref{e8}). The homomorphism
$\widetilde{\eta}$ evidently is a $G$--connection on $E_H$.

Let ${\mathcal K}(\eta)\, \in\, H^0(X,\, \Omega^2_X\otimes \text{ad}(E_H))$
be the curvature of the connection $\eta$, where $\Omega^2_X\,=\, \bigwedge^2 T^*X$. For
any $w\, \in\, T_xX$, let
\begin{equation}\label{v3}
i_w({\mathcal K}(\eta)(x))\, \in\, (T^*X)_x\otimes \text{ad}(E_H)_x
\,=\, (T^*X\otimes \text{ad}(E_H))_x
\end{equation}
be the contraction of ${\mathcal K}(\eta)(x)\, \in\, (\Omega^2_X\otimes \text{ad}(E_H))_x$
by the tangent vector $w\,\in\, T_xX$.

\begin{lemma}\label{lem2}
The connection $\eta$ on $E_H$ is strongly adapted to the above constructed $G$--connection
$\widetilde{\eta}$ if and only if for all $v\, \in\, \mathfrak g$ and $x\in\, X$,
\begin{equation}\label{z}
i_{d'\rho(x,v)}({\mathcal K}(\eta)(x))\,=\, 0\, ,
\end{equation}
where $d'\rho$ is defined in \eqref{dpr} (see \eqref{v3} for the contraction).
\end{lemma}

\begin{proof}
{}From the construction of $\widetilde{\eta}$ in \eqref{te} it follows immediately that
the condition in \eqref{Jc} holds. We need to show that \eqref{a1} holds
if and only if \eqref{z} holds.

To prove this, we recall a construction of the curvature
${\mathcal K}(\eta)$. Given a point $x\, \in\, X$ and holomorphic tangent
vectors $v,\, w\, \in\, T_xX$, extend $v,\, w$ to vector fields
$\widetilde{v},\, \widetilde{w}$ of type $(1,\, 0)$ on some open neighborhood of
the point $x$. Let $\widehat{v}\,=\, \eta(\widetilde{v})$
and $\widehat{w}\,=\, \eta(\widetilde{w})$ be the horizontal lifts
of $\widetilde{v}$ and $\widetilde{w}$ respectively, for the connection $\eta$. Then
$$
{\mathcal K}(\eta)(x)(v,w)\,=\, ([\widehat{v},\, \widehat{w}]_{\rm Vert})\vert_{p^{-1}(x)}
\, ,
$$
where $[\widehat{v},\, \widehat{w}]_{\rm Vert}$ is the component of
the Lie bracket $[\widehat{v},\, \widehat{w}]$ in the vertical direction
(for the projection $p$). We note that
the section $([\widehat{v},\, \widehat{w}]_{\rm Vert})\vert_{p^{-1}(x)}$ of
$T_{E_H/X}$ over $p^{-1}(x)$ is $H$--invariant and hence it defines an element
of the fiber $\text{ad}(E_H)_x$ over $x$; recall that $\text{ad}(E_H)$ is
identified with $(T_{E_H/X})/H$. The element $([\widehat{v},\,
\widehat{w}]_{\rm Vert})\vert_{p^{-1}(x)}\,\in\, \text{ad}(E_H)_x$ does not depend
on the choice of the extensions $\widetilde{v}$ and $\widetilde{w}$ of $v$ and
$w$ respectively. From this description of ${\mathcal K}(\eta)$
it follows immediately that \eqref{a1} holds if and only if \eqref{z} holds.
\end{proof}

{}From the proof of Lemma \ref{lem2} we have the following:

\begin{corollary}\label{cor1}
The connection $\eta$ on $E_H$ is adapted to the above constructed $G$--connection
$\widetilde{\eta}$ if and only if the condition in \eqref{z} holds. In other words,
the connection $\eta$ on $E_H$ is strongly adapted to
$\widetilde{\eta}$ if $\eta$ is adapted to $\widetilde{\eta}$.
\end{corollary}

Take a $\mathbb C$--linear map
\begin{equation}\label{v1}
\varphi_0\, :\, {\mathfrak g}\, \longrightarrow\, H^0(X,\, \text{ad}(E_H))\, .
\end{equation}
For any $v\, \in\, \mathfrak g$, the section $\varphi_0(v)\, \in\, H^0(X,\,
\text{ad}(E_H))$ defines a holomorphic vertical tangent vector field on
$E_H$ for the projection $p$. This vertical tangent vector field on $E_H$
will be denoted by ${\varphi}(v)$. Let $U\, \subset\, X$ be an open subset
and $V$ a $C^\infty$ vector field on $U$ of type $(1,\, 0)$.
Let $V'\,=\, \eta(V)$ be the horizontal lift of
$V$ on $p^{-1}(U)$ for the holomorphic connection $\eta$ on $E_H$.
Let $f_0$ be any $C^\infty$
function on $U$. Then $V'(f_0\circ p)$ is a $H$--invariant function on $p^{-1}(U)$, and
hence
\begin{equation}\label{Q1}
{\varphi}(v)(V'(f_0\circ p))\,=\, 0\, .
\end{equation}
On the other hand,
\begin{equation}\label{Q2}
{\varphi}(v)(f_0\circ p)\,=\, 0
\end{equation}
because ${\varphi}(v)$ is a vertical
vector field. From \eqref{Q1} and \eqref{Q2} we conclude that
$$
[{\varphi}(v),\, V'](f_0\circ p)\,=\, 0\, .
$$
In other words,
\begin{equation}\label{ev}
[{\varphi}(v),\, V']\,=\, [{\varphi}(v),\, V']_{\rm Vert}\, ,
\end{equation}
where $[{\varphi}(v),\, V']_{\rm Vert}$ is the vertical component of
$[{\varphi}(v),\, V']$.
The vector field $[{\varphi}(v),\, V']$ is
$H$--invariant because both ${\varphi}(v)$ and $V'$ are $H$--invariant.
If $f_1$ is a $C^\infty$ function on $U$, then note that
$$
[{\varphi}(v),\, (f_1\circ p)\cdot V']\,=\, (f_1\circ p)\cdot
[{\varphi}(v),\, V']
$$
because ${\varphi}(v)(f_1\circ p)\,=\, 0$. Clearly, the vector field $(f_1\circ
p)\cdot V'$ is the horizontal lift of the vector field $f_1\cdot V$ on $U$ for the
connection $\eta$. From these observations we conclude that there is a homomorphism
\begin{equation}\label{v2}
\widetilde{\varphi}\, :\, {\mathfrak g}\otimes_{\mathbb C} TX\, \longrightarrow\,
\text{ad}(E_H)
\end{equation}
that sends $v\otimes w\,\in\, {\mathfrak g}\otimes T_xX$ to
$[{\varphi}(v),\, V'](x)$, where $V'\,=\, \eta(V)$ is the horizontal lift, with
respect to the connection $\eta$, of a vector field $V$ defined on a neighborhood of the
point $x\,\in\, X$ with $V(x)\,=\, w$. Note that $[{\varphi}(v),\, V'](x)$ does
not depend on the choice of the extension $V$ of $w$.

The contraction in \eqref{v3} produces a homomorphism
\begin{equation}\label{Pi}
\Pi\, :\, {\mathfrak g}\otimes_{\mathbb C} TX\, \longrightarrow\,
\text{ad}(E_H)
\end{equation}
that sends $v\otimes w\,\in\, {\mathfrak g}\otimes T_xX$ to
$$
i_w i_{d'\rho(x,v)}({\mathcal K}(\eta)(x))\, \in\,
\text{ad}(E_H)_x\, ,
$$
which is the contraction of $i_{d'\rho(x,v)}({\mathcal K}(\eta)(x))\, \in\,
(T^*X)_x\otimes \text{ad}(E_H)_x$ (see \eqref{dpr}, \eqref{v3}) by the
tangent vector $w\,\in\, T_xX$.

\begin{theorem}\label{thm1}
Let $X$ be a complex manifold equipped with a holomorphic action of $G$ and
$E_H$ a holomorphic principal $H$--bundle on $X$ equipped with a holomorphic
connection $\eta$. Then there is a $G$--connection $h$ on $E_H$ such that $\eta$
is adapted to $h$ if and only if there is a homomorphism $\varphi_0$ as in \eqref{v1}
such that the homomorphism $\widetilde{\varphi}$ in \eqref{v2} coincides with
the homomorphism $-\Pi$, where $\Pi$ is constructed in \eqref{Pi}.
\end{theorem}

\begin{proof}
Let $h\, :\, {\mathfrak g}\, \longrightarrow\,
H^0(X,\, \text{At}_\rho(E_H))$ be a $G$--connection
on $E_H$ such that $\eta$ is adapted to $h$. For any $v\, \in\, \mathfrak g$,
consider $$J\circ h(v) - \eta(v')\, \in\, H^0(X,\, \text{At}(E_H))\, ,$$
where $J$ is the homomorphism in \eqref{e9} and $v'$ is the holomorphic vector field on 
$X$ defined by $x\, \longmapsto\,d'\rho(x,v)$
(see \eqref{dpr}). Note that $dp\circ J\circ h(v)\,=\, v'$, where $dp$ is the homomorphism
in \eqref{e2}. Therefore, we have
$$
J\circ h(v) - \eta(v')\, \in\, H^0(X,\, \text{ad}(E_H))\, \subset\,
H^0(X,\, \text{At}(E_H))
$$
(see \eqref{e9}). Now define
$$
\varphi_0\, :\, {\mathfrak g}\, \longrightarrow\, H^0(X,\, \text{ad}(E_H))\, ,\ \
v\,\longmapsto\, J\circ h(v) - \eta(v')\, .
$$
We will show that the homomorphism $\widetilde{\varphi}$ in \eqref{v2} for this
$\varphi_0$ coincides with the homomorphism $-\Pi$.

Take any $v\, \in\, \mathfrak g$. Given any $x\, \in\, X$ and any $w\, \in\, T_xX$,
let $V$ be any $C^\infty$ vector field of type $(1,\, 0)$,
defined on an open neighborhood of $x\, \in\, X$, such that
$$
[v',\, V]\, =\, 0\, .
$$

Since $\eta$ is adapted to $h$, the Lie bracket $[J\circ h(v),\, \eta(V)]$ lies in the horizontal 
subbundle $\eta(TX)\,\subset\, TE_H$. In other words, the vertical component of $[J\circ h(v),\, 
\eta(V)]$ vanishes identically.

The Lie bracket $[\eta(v'),\, \eta(V)]$ is vertical because
$dp([\eta(v'),\, \eta(V)])\,=\, [v',\, V]\,=\,0$. From \eqref{ev} we know that
the Lie bracket $[{\varphi}(v),\, \eta(V)]$ is vertical, where
${\varphi}(v)$ is the vertical vector field corresponding to $\varphi_0(v)
\,\in\, H^0(X,\, \text{ad}(E_H))$. This and the
fact that $[\eta(v'),\, \eta(V)]$ is vertical together imply that
\begin{equation}\label{s1}
[{\varphi}(v)+\eta(v'),\, \eta(V)]\,=\, [J\circ h(v),\, \eta(V)]
\end{equation}
is vertical. But it was shown above that the vertical
component of $[J\circ h(v),\, \eta(V)]$ vanishes identically. Hence we conclude that
$$
[J\circ h(v),\, \eta(V)]\,=\, 0\, .
$$
Consequently, we have
\begin{equation}\label{ii}
[{\varphi}(v),\, \eta(V)]\,=\,- [\eta(v'),\, \eta(V)]
\end{equation}
for all $v\, \in\, \mathfrak g$. Since $[{\varphi}(v),\, \eta(V)]\,=\,
\widetilde{\varphi}(v\otimes V)$ and $[\eta(v'),\, \eta(V)]\,=\,
\Pi(v\otimes V)$, from \eqref{ii} it follows that 
$$
\widetilde{\varphi}\,=\, -\Pi\, .
$$

To prove the converse, take any homomorphism $\varphi_0$ as in \eqref{v1} such that
\begin{equation}\label{iii}
\widetilde{\varphi}\,=\, -\Pi\, .
\end{equation}
Now define a $G$--connection
$$
h\, :\, {\mathfrak g}\, \longrightarrow\, H^0(X,\, \text{At}_\rho(E_H))\, ,
v\, \longmapsto\, (\varphi_0(v)+\eta(v'),\, X\times\{v\})\, .
$$
We will show that $\eta$ is adapted to $h$.

Let $V$ be a $C^\infty$ vector field of type $(1,\, 0)$ defined on an
open subset $U\,\subset\, X$. Take any $v\, \in\,\mathfrak g$. The Lie
bracket $[{\varphi}(v),\, \eta(V)]$ is vertical (see \eqref{ev}), where
$\varphi(v)$, as before, is the vertical vector field for the projection $p$ corresponding 
to the section ${\varphi}_0(v)$ of $\text{ad}(E_H)$. 
We have
$$
\widetilde{\varphi}(v\otimes V)\,=\, [{\varphi}(v),\, \eta(V)]\, ,
$$
and $\Pi(v\otimes V)$ is the vertical component of $[\eta(v'),\, \eta(V)]$. Consequently,
from \eqref{iii} and the definition of $h$ it follows that the vertical component of
$[J\circ h(v),\, \eta(V)]$ vanishes. This implies that $\eta$ is adapted to $h$.
\end{proof}

Let $h\,:\, {\mathfrak g}\, \longrightarrow\, H^0(X,\, \text{At}_{\rho}(E_H))$ be a
$G$--connection on $E_H$. Take any section
$$
\theta\, \in\, C^\infty(X, \, \text{At}(E_H)^{\otimes a}\otimes (\text{At}(E_H)^*)^{\otimes
b})\, ,
$$
where $a$ and $b$ are nonnegative integers. 
Note that $\theta$ defines a $H$--invariant section of the vector bundle
$(TE_H)^{\otimes a}\otimes (T^*E_H)^{\otimes b}$ on $E_H$; this section of
$(TE_H)^{\otimes a}\otimes (T^*E_H)^{\otimes b}$ will be denoted by $\Theta$.
We say that $\theta$ is preserved by $h$ if
$$
L_{J\circ h(v)} \Theta\, =\, 0 \ \ \ \forall \ \ v\, \in\, \mathfrak g\, ,
$$
where $L_{J\circ h(v)}$ is the Lie derivative with respect to the vector
field $J\circ h(v)$ on $E_H$ (the homomorphism $J$ is constructed
in \eqref{e9}).

If $h$ is the $G$--connection associated to a $G$--action $\rho_E$ on $E_H$, then it
is straight-forward to check that $\theta$ is preserved by $h$ if and only if
the section $\Theta$ is preserved by the action $\rho_E$ on $E_H$.

\section{Holomorphic foliations and strongly adapted connections}\label{se-f}

As before, $X$ is a complex manifold. Let
$$
{\mathcal F}\, \subset\, TX
$$
be a holomorphic foliation on $X$, which means that $\mathcal F$ is a holomorphic
subbundle of $TX$ such that for any two sections $s$ and $t$ of $\mathcal F$ defined over
some open subset of $X$, the Lie bracket $[s,\, t]$ is also a section 
of $\mathcal F$ \cite{La}. Let $E_H$ be a holomorphic principal $H$--bundle on $X$.

Consider the Atiyah exact sequence for $E_H$ in \eqref{e2}. Define
$$
\text{At}_{\mathcal F}(E_H)\, :=\, (dp)^{-1}({\mathcal F})\, \subset\,
\text{At}(E_H)\, .
$$
So, from \eqref{e2} we have the short exact sequence of holomorphic vector bundles
\begin{equation}\label{e2p}
0\, \longrightarrow\, \text{ad}(E_H)\, \longrightarrow\, \text{At}_{\mathcal F}(E_H)\,
\stackrel{\widetilde{dp}}{\longrightarrow} \, {\mathcal F} \, \longrightarrow\, 0\, ,
\end{equation}
where $\widetilde{dp}$ is the restriction of $dp$ to $\text{At}_{\mathcal F}(E_H)$.
A \textit{holomorphic partial connection} on $E_H$ is a homomorphism
$$
D\, :\, {\mathcal F}\, \longrightarrow\,\text{At}_{\mathcal F}(E_H)
$$
such that $\widetilde{dp}\circ D\,=\, \text{Id}_{\mathcal F}$ \cite{La}.

Given such a holomorphic partial connection $D$, the homomorphism
$$
\bigwedge\nolimits^2 {\mathcal F}\, \longrightarrow\,\text{ad}(E_H)\, ,\ \
v\otimes w - w\otimes v \, \longmapsto\, 2([D(v),\, D(w)]- D([v,\, w]))\, ,
$$
where $v$ and $w$ are locally defined holomorphic sections of ${\mathcal F}$, produces a
holomorphic section of $(\bigwedge^2 {\mathcal F}^*)\otimes \text{ad}(E_H)$. This
holomorphic section of $(\bigwedge^2 {\mathcal F}^*)\otimes \text{ad}(E_H)$ is called the
\textit{curvature} of the partial connection $D$. A holomorphic partial connection
is called \textit{flat} if its curvature vanishes identically.

Let $\eta\, :\, TX\, \longrightarrow\,\text{At}(E_H)$ be a holomorphic connection
on the principal $H$--bundle $E_H$. As before, the curvature of $\eta$ will be denoted
by ${\mathcal K}(\eta)$. 
Let $D\, :\, {\mathcal F}\, \longrightarrow\,\text{At}_{\mathcal F}(E_H)$ be a
flat holomorphic partial connection on $E_H$. 

The connection $\eta$ is said to be \textit{strongly adapted} to $D$ if
\begin{itemize}
\item the restriction $\eta\vert_{\mathcal F}\, :\, {\mathcal F}\, \longrightarrow\,
\text{At}(E_H)$ coincides with $D$, and

\item for any $x\,\in\, X$ and $w\, \in\, {\mathcal F}_x$, the contraction
$$i_w {\mathcal K}(\eta)(x)\,\in\, T^*_xX\otimes \text{ad}(E_H)_x$$ vanishes.
\end{itemize}

\begin{corollary}\label{cor3}
Suppose that $\mathcal F$ is given by a holomorphic action $\rho$ of a connected 
complex Lie group $G$ on $X$ (so the leaves of $\mathcal F$ are the orbits of $G$), 
and also assume that $D$ is given by a $G$--action $\rho_E$ on $E_H$ (so the tangent 
spaces to the leaves in $E_H$ are the horizontal subspaces). Then $\eta$ is strongly 
adapted to $D$ if and only if $\eta$ is strongly adapted to the $G$--connection on 
$E_H$ given by $\rho_E$.
\end{corollary}

\begin{proof}
The above condition that $\eta\vert_{\mathcal F}\,=\, D$ is equivalent to the
condition that the $G$--connection $\widetilde\eta$ constructed in \eqref{te}
from $\eta$ coincides with the $G$--connection on
$E_H$ given by the above $G$--action $\rho_E$. Therefore, the result
follows from Lemma \ref{lem2}.
\end{proof}


\end{document}